\documentclass[preprint,1p,times,sort&compress]{elsarticle}
\usepackage{amsmath,amsfonts,mathrsfs}
\usepackage{amssymb}
\usepackage{color}
\usepackage{amsthm}
\newtheorem{theorem}{Theorem}[section]
\newtheorem{lemma}{Lemma}[section]

\newdefinition{rem}{Remark}

\numberwithin{equation}{section}

\journal{}

\begin{document}
\begin{frontmatter}
\title{Boundedness of solutions for the reversible system with low regularity in time}

\author[SDU,SDUW]{Jing Li
\corref{cor}}
\cortext[cor]{Corresponding author.}
\ead{xlijing@sdu.edu.cn}
\address[SDU]{School of Mathematics, Shandong University, Jinan 250100, P.R. China}
\address[SDUW]{School of Mathematics and Statistics, Shandong University, Weihai 264209, P.R. China}
\footnote{The work was supported in part by National
Nature Science Foundation of China (11601277).}
\begin{abstract}
In the present paper, it is proved that all solutions are bounded  for the reversible system
$\ddot{x}+\sum_{i=0}^{l}b_{i}(t)x^{2i+1}\dot{x}+x^{2n+1}+\sum_{i=0}^{n-1}a_{i}(t)x^{2i+1}=0,
0\leq l\leq [\frac{n}{2}]-1,t\in\mathbb{T}^{1}=\mathbb{R}/\mathbb{Z},$
where $a_{i}(t)\in C^{1}(\mathbb{T}^{1})\;([\frac{n-1}{2}]+1\leq i\leq n-1),$
 $a_{j}(t)\in L^{1}(\mathbb{T}^{1})\;(0\leq j\leq [\frac{n-1}{2}])$ and
 $b_{k}(t)\in C^{1}(\mathbb{T}^{1})\;(0\leq k\leq l).$

\end{abstract}


\end{frontmatter}

\section{Introduction}

The boundedness of all solutions for the differential equation
\begin{equation}\label{1}
\ddot{x}+f(x, t)\dot{x}+g(x, t)=0,\;\;x\in \mathbb{R}
\end{equation}
has been widely and deeply investigated by many authors since 1940's. The boundedness of
solutions depends heavily on the structure of \eqref{1}.

(i) When $f(x, t)\equiv 0,$ \eqref{1} is a Hamiltonian system. Let
$g(x, t)=x^{2n +1}+\sum_{j=0}^{2n}P_{j}(t)x^{j},$
where $P_{j}(t)$'s are of period $1 $. It has been proved by Dieckerhoff-Zehnder in \cite{Dieckerhoff-Zehnder1987}
that all solutions of \eqref{1} are bounded in $t\in \mathbb{R}$ if $P_{j}(t)\in {C}^{\infty}.$ The smoothness
of $P_{j}(t)$'s has been recently reduced to ${C}^{\gamma}$ with $0<\gamma<1-\frac{1}{n}$ in \cite{Yuan2017}. See
\cite{Dieckerhoff-Zehnder1987,Laederich-Levi1991,Yuan1995,Yuan1998,Yuan2000,Yuan2017} for more details.

(ii) When $f(x, t)\not\equiv 0$, \eqref{1} is dissipative with more appropriate conditions. A compact
absorbing domain in the phase space can be constructed such that all solutions of \eqref{1} always go into this domain for
$t\geq t_{0}.$ See \cite{Graef,Levinson1943,Reuter1951}, for example.

(iii) When $f(x, t)$ and $g(x, t)$ are odd in $x$ and even in $t,$ \eqref{1} is neither Hamiltonian nor dissipative.
In this case, \eqref{1} is actually in the class of so-called  reversible systems. After the Kolmogorov-Arnold-Moser (KAM) theory was established,
Arnold \cite{Arnold1984,Arnold1986} and Moser \cite{Moser} among others proposed the study of the existence of invariant tori for the
reversible systems by KAM technique, i.e. to establish KAM theory for the reversible system. See \cite{Sevryuk} for more details.

Let
$
f(x,t)=\sum_{i=0}^{l}b_{i}(t)x^{2i+1}$, $0\leq l\leq [\frac{n}{2}]-1,$ and $g(x,t)=x^{2n+1}+\sum_{i=0}^{n-1}a_{i}(t)x^{2i+1},
$
where $a_{i}(t)$'s and $b_{i}(t)$'s are even and of period 1. Then \eqref{1} is a simple reversible system, and it has
non-trivial dynamical behaviors. The KAM theory for reversible systems \cite{Moser,Sevryuk} deals with some integrable systems
with small reversible perturbations. When $a_{i}(t)\equiv b_{i}(t)\equiv 0,$ \eqref{1} is indeed integrable. However, \eqref{1}
can not be regarded as an integrable system with small reversible perturbation when $a_{i}(t)\not\equiv 0,$ $b_{i}(t)\not\equiv 0.$
Following \cite{Dieckerhoff-Zehnder1987}, \eqref{1} can be transformed into a new system which consists of an integrable system with
a small reversible one around the infinity by a number of so-called involution transforms. In that direction, Liu in \cite{Liu1991} proved that all solutions are
bounded for
$\ddot{x}+b x \dot{x}+cx^{2n+1}=p(t),$
where $b, c$ are positive constants and $p(t)$ is a continuous 1-period function. This result was generalized to the more general case where
$f(x,t)=\sum_{i=0}^{l}b_{i}(t)x^{2i+1},0\leq l\leq [\frac{n}{2}]-1$, and $g(x,t)=x^{2n+1}+\sum_{i=0}^{n-1}a_{i}(t)x^{2i+1}
$, where $a_{i}(t)\in C^{2}$ and $b_{i}(t)\in {C}^{2}.$
Later, the smoothness of $a_{i}(t)$ and $b_{i}(t)$ in \cite{Piao2008, Rong2001} were furthermore relaxed to
$a_{i}(t),\;b_{i}(t)\in {C}^{1+Lip}.$
In the present paper, we relax the smoothness to $a_{i}(t), b_{i}(t)\in {C}^{1}.$ More exactly, we have the following theorem.
\begin{theorem}\label{thm1}
Consider
\begin{equation}\label{*}
\ddot{x}+\sum_{i=0}^{l}b_{i}(t)x^{2i+1}\dot{x}+x^{2n+1}+\sum_{i=0}^{n-1}a_{i}(t)x^{2i+1}=0,\;\;t\in\mathbb{T}^{1}=\mathbb{R}/\mathbb{Z},
\end{equation}
where $0\leq l\leq [\frac{n}{2}]-1,$ and
\begin{itemize}
  \item $b_{i}(t)\in C^{1}(\mathbb{T}^{1}), $ $b_{i}(-t)=b_{i}(t), 0\leq i \leq l,$
  \item $a_{i}(t)\in C^{1}(\mathbb{T}^{1}), [\frac{n-1}{2}]+1\leq i\leq n-1;$ $a_{i}(t)\in L^{1}(\mathbb{T}^{1}),0\leq i\leq [\frac{n-1}{2}];$
          $a_{i}(-t)=a_{i}(t), 0\leq i \leq n-1.$
\end{itemize}
Then all solutions of \eqref{*} are bounded, i.e. the solution $(x(t), \dot{x}(t))$ with initial values $(x(0), y(0))$ exists for all
$t\in \mathbb{R}$ and
$$\sup_{t\in \mathbb{R}}(|x(t)|+|\dot{x}(t)|)\leq C_{(x(0), y(0))},$$
where $C_{(x(0), y(0))}>0$ is a constant depending on the initial values $(x(0), y(0)).$
\end{theorem}
\begin{rem}
It is still open whether the smoothness of $a_{i}(t)$ and $b_{i}(t)$ can be relaxed to $C^{\kappa}$ $(0<\kappa<1)$ as in \cite{Yuan2017}.
\end{rem}
\section{Action-Angle variables}
Consider
\eqref{*}.
First, rescale $x \rightarrow Ax,$ where $A$ is a large constant. Then \eqref{*}  can be written as a system:
\begin{equation}\label{2.2}
\dot{x} =A^{n}y,\;\;
    \dot{y}=-A^{n}x^{2n+1}-\sum_{i=0}^{n-1}A^{2i-n}a_{i}(t)x^{2i+1}-\sum_{i=0}^{l}A^{2i+1}b_{i}(t)x^{2i+1}y.
\end{equation}
First of all, we consider an unperturbed Hamiltonian system
\begin{equation}\label{2.4}
\frac{dx}{dt}=y=\frac{\partial H_{0}}{\partial y},\;\;
   \frac{dy}{dt}=-x^{2n+1}=-\frac{\partial H_{0}}{\partial x},
\end{equation}
where
$H_{0}(x,y)=\frac{1}{2}y^{2}+\frac{1}{2n+2}x^{2n+2}.$
Assume $(S(t), C(t))$ is the solution of \eqref{2.4} with the initial condition $(S(0), C(0))=(0, 1).$
Clearly, this solution is periodic. Let $T_{0}$ be its minimal positive period. It follows from \eqref{2.4}
that $S(t)$ and $C(t)$ satisfy the following properties:
\begin{itemize}
  \item [(i)] $S(t),$ $C(t)\in C^{\omega}(\mathbb{T})$ \; (ii) $S(t+T_{0})=S(t),$ $C(t+T_{0})=C(t);$
  \item [(iii)] $\dot{S}(t)=C(t),$ $\dot{C}(t)=-S^{2n+1}(t);$ \;(iv) $S^{2n+2}(t)+(n+1)C^{2}(t)=n+1;$
  \item [(v)] $C(-t)=C(t),$ $S(-t)=-S(t).$
\end{itemize}
Following \cite{Dieckerhoff-Zehnder1987}, we define a diffeomorphism
$\psi_{0}: \mathbb{R}^{+}\times \mathbb{T}^{1} \rightarrow \mathbb{R}^{2}\backslash \{0\}$.
Let
$
\psi_{0}:
             x=c^{\alpha}\rho^{\alpha}S(\theta T_{0}), \;
             y=c^{\beta}\rho^{\beta}C(\theta T_{0}),
$
where $\alpha=\frac{1}{n+2},$ $\beta=1-\alpha,$ $c=\frac{1}{\beta T_{0}}.$ By a simple calculation,
we have $\mid\det \frac{\partial (x,y)}{\partial (\rho, \theta)}\mid=1.$ Thus $\psi_{0}$ is symplectic.
So by $\psi_{0},$ \eqref{2.2} is changed into
\begin{equation}\label{2.7}
\frac{d\rho}{dt}=f_{1}(\rho, \theta, t)+f_{2}(\rho, \theta, t), \;\;
  \frac{d\theta}{dt}=d A^{n} \rho^{2\beta-1}+g_{1}(\rho, \theta, t)+g_{2}(\rho, \theta, t),
\end{equation}
where $d=\beta c^{2\beta},$ and
\begin{eqnarray}
\label{2.8}f_{1}(\rho, \theta, t)&=&-\sum_{i=[\frac{n-1}{2}]+1}^{n-1}A^{2i-n}a_{i}(t)T_{0}c^{2(i+1)\alpha}\rho^{2(i+1)\alpha}C(\theta T_{0})S^{2i+1}(\theta T_{0})\nonumber\\
                       &&-\sum_{i=0}^{l}A^{2i+1}b_{i}(t)T_{0}c^{(2i+1)\alpha+1}\rho^{(2i+1)\alpha+1}C^{2}(\theta T_{0})S^{2i+1}(\theta T_{0}),\\
\label{2.9}
f_{2}(\rho, \theta, t)&=&-\sum_{i=0}^{[\frac{n-1}{2}]}A^{2i-n}a_{i}(t)T_{0}c^{2(i+1)\alpha}\rho^{2(i+1)\alpha}C(\theta T_{0})S^{2i+1}(\theta T_{0}),\\
\label{2.10}
g_{1}(\rho, \theta, t)&=&\alpha\sum_{i=[\frac{n-1}{2}]+1}^{n-1} A^{2i-n}a_{i}(t) c^{2(i+1)\alpha}\rho^{2(i+1)\alpha -1}S^{2(i+1)}(\theta T_{0})\nonumber\\
                       &&+\alpha\sum_{i=0}^{l}A^{2i+1}b_{i}(t) c^{2(i+1)\alpha+1}\rho^{2(i+1)\alpha}C(\theta T_{0})S^{2(i+1)}(\theta T_{0}),\\
\label{2.11}
g_{2}(\rho, \theta, t)&=&\alpha\sum^{[\frac{n-1}{2}]}_{i=0} A^{2i-n}a_{i}(t) c^{2(i+1)\alpha}\rho^{2(i+1)\alpha -1}S^{2(i+1)}(\theta T_{0}).
\end{eqnarray}
Recall $C(-t)=C(t),$ $S(-t)=-S(t),$ $a_{i}(-t)=a_{i}(t)$ and $b_{i}(-t)=b_{i}(t).$ 
So we have
\begin{eqnarray}
&& \label{2.12} f_{1}(\rho, -\theta, t)=-f_{1}(\rho, \theta, t),  f_{2}(\rho, -\theta, -t)=-f_{2}(\rho, \theta, t),
\;f_{1}(\rho, -\theta, -t)=-f_{1}(\rho, \theta, t),\;\\
&& \label{2.14} g_{1}(\rho, -\theta, t)=g_{1}(\rho, \theta, t),\;g_{1}(\rho, \theta, -t)=g_{1}(\rho, \theta, t),
\;g_{2}(\rho, -\theta, -t)=g_{2}(\rho, \theta, t).
\end{eqnarray}
In addition, by \eqref{2.8}-\eqref{2.11}, we have
$f_{1}=O_{1}(A^{n-1}), f_{2}=O(A^{-1}), g_{1}=O_{1}(A^{n-1}), g_{2}=O(A^{-1}),
$
where $f(\rho, \theta, t, A)=O_{1}(A^{\Gamma})$ means
 \begin{eqnarray}
&&\sup_{\begin{array}{c}
(\rho, \theta, t)\in D_{4-CA^{-1}}\times  \mathbb{T} ^{1}
\end{array}}
\Big |\sum_{p+q=k}\frac{\partial^{k+1}f}{\partial \rho^{p}\partial \theta^{q}\partial t}\Big |\leq C_{k} A^{\Gamma},\;\;A\gg1,\;\;k\in\mathbb{Z}_{+}
\end{eqnarray}
with constant $C_{k}$ depending on $k$ and
 \begin{eqnarray}
&&D_{s}=\{(\rho, \theta)\in \mathbb{R}\times \mathbb{T}^{1}: 1\leq \rho \leq s, \theta \in \mathbb{T}^{1}\},
\end{eqnarray}
and $f(\rho, \theta, t, A)=O(A^{\Gamma})$ means
\begin{eqnarray}
\sup_{\begin{array}{c}
(\rho, \theta, t)\in D_{4-CA^{-1}}\times  \mathbb{T} ^{1}
\end{array}}
\Big |\sum_{p+q=k}\frac{\partial^{k}f}{\partial \rho^{p}\partial \theta^{q}}\Big |\leq C_{k} A^{\Gamma},\;\;A\gg1,\;\;k\in\mathbb{Z}_{+}.
\end{eqnarray}
\section{Coordinate changes}
\begin{lemma}\label{3.1}
There exists a diffeomorphism
$\Psi^{1}: \rho=\mu+U_{1}(\mu, \phi, t),\theta=\phi$
such that $\Psi^{1}(D_{4-C_{0}A^{-1}})\subset D_{4}$ with a constant $C_{0}>0$ and \eqref{2.7} is changed into
$
   \frac{d\mu}{dt}=f_{1}^{(1)} (\rho, \theta, t)+f_{2}^{(1)}(\rho, \theta, t),\;
   \frac{d\phi}{dt}=d A^{n} \mu^{2\beta-1}+g_{1}^{(1)}(\rho, \theta, t)+g_{2}^{(1)}(\rho, \theta, t),
$
where $f_{1}^{(1)}, f_{2}^{(1)}$ and $g_{1}^{(1)}, g_{2}^{(1)}$ satisfy \eqref{2.12} and \eqref{2.14}, respectively,
and
\begin{eqnarray}
\label{3.2}f_{1}^{(1)}=O_{1}(A^{n-2}),\;f_{2}^{(1)}=O(A^{-1}),\;g_{1}^{(1)}=O_{1}(A^{n-1}),\;g_{2}^{(1)}=O(A^{-1}),\;(\mu, \phi)\in D_{4-C_{0}A^{-1}}.
\end{eqnarray}
\end{lemma}
\begin{proof}
Set
$\Phi^{1}:\;\;\mu=\rho+V_{1}(\rho, \theta, t),\;\;\phi=\theta.$
Under $\Phi^{1},$ we have
$\frac{d\mu}{dt}=\frac{d\rho}{dt}+\frac{\partial V_{1}}{\partial \rho}\frac{d\rho}{dt}
+\frac{\partial V_{1}}{\partial \theta}\frac{d \theta}{dt}+\partial_{t}V_{1}.$
By \eqref{2.7},
$
\frac{d\mu}{dt}=f_{1}+f_{2}+\frac{\partial V_{1}}{\partial \rho}(f_{1}+f_{2})
+\frac{\partial V_{1}}{\partial \theta}(d\rho ^{2\beta-1}A^{n}+g_{1}+g_{2})+\frac{\partial V_{1}}{\partial t}.
$
Since $f_{1}(\rho, -\theta, t)=-f_{1}(\rho, \theta, t),$ $[f_{1}]=\int_{\mathbb{T}^{1}}f_{1}(\rho, \theta, t)d \theta=0.$
So by setting
$d A^{n}\rho^{2\beta-1}\frac{\partial V_{1}}{\partial \theta}+f_{1}(\rho, \theta, t)=0,$
we can get
$V_{1}(\rho, \theta, t)=-\int_{0}^{\theta}\frac{f_{1}(\rho, s, t)}{d A^{n}\rho^{2\beta-1}}ds.$
Recall $f_{1}(\rho, -\theta, t)=-f_{1}(\rho, \theta, t)$ and $f_{1}=O_{1}(A^{n-1}).$ We have
\begin{eqnarray}
\label{3.5.1} V_{1}(\rho, -\theta, t)=V_{1}(\rho, \theta, t)=V_{1}(\rho, \theta, -t), V_{1}(\rho, \theta, t)=O_{1}(A^{-1}),\;\;(\rho, \theta, t)\in D_{4}\times \mathbb{T}^{1}.
\end{eqnarray}
By the implicit function Theorem, we have that there exists the inverse of $\Phi^{1}$, say $\Psi^{1},$
which can be written as
$\Psi^{1}=(\Phi^{1})^{-1}:\;\;\rho=\mu+U_{1}(\mu, \phi, t),\;\;\theta=\phi,$
where
\begin{eqnarray}
\label{3.7.1} U_{1}(\mu, -\phi, t)=U_{1}(\mu, \phi, t)=U_{1}(\mu, \phi, -t),U_{1}=O_{1}(A^{-1}),\;\;(\mu, \phi, t)\in D_{4-C_{0}A^{-1}}\times \mathbb{T}^{1}.
\end{eqnarray}
Let
 $f_{1}^{(1)}
=\frac{\partial V_{1}(\mu+U_{1}(\mu, \phi, t), \phi, t)}{\partial \rho } f_{1}(\mu+U_{1}(\mu, \phi, t), \phi, t)
+\frac{\partial V_{1}(\mu+U_{1}(\mu, \phi, t), \phi, t)}{\partial \theta } g_{1}(\mu+U_{1}(\mu, \phi, t), \phi, t),$
$
f_{2}^{(1)}=f_{2}(\mu+U_{1}(\mu,\phi,t),\phi,t)+\frac{\partial V_{1}(\mu+U_{1}(\mu,\phi,t),\phi,t)}{\partial \rho }f_{2}
+\frac{\partial V_{1}(\mu+U_{1}(\mu,\phi,t),\phi,t)}{\partial\theta}g_{2}(\mu+U_{1}(\mu,\phi,t),\phi,t)
+\frac{\partial V_{1}(\mu+U_{1}(\mu,\phi,t),\phi,t)}{\partial t},$
$
g_{1}^{(1)}=g_{1}(\mu+U_{1}(\mu, \phi, t), \phi, t)-d A^{n}\mu^{2\beta-1}+d A^{n}(\mu+U_{1}(\mu, \phi, t))^{2\beta-1},$
$g_{2}^{(1)}=g_{2}(\mu+U_{1}(\mu, \phi, t), \phi, t).
$
Using \eqref{2.12}, \eqref{2.14}, \eqref{3.5.1} and \eqref{3.7.1}, we have that $f_{1}^{(1)}$, $f_{2}^{(1)}$ and $g_{1}^{(1)}$, $g_{2}^{(1)}$
satisfy \eqref{2.12} and \eqref{2.14}, respectively. Using \eqref{3.5.1} and \eqref{3.7.1}, we have
$
f_{1}^{(1)}=O_{1}(A^{n-2}),f_{2}^{(1)}=O(A^{-1}),g_{1}^{(1)}=O_{1}(A^{n-1}),g_{2}^{(1)}=O(A^{-1}),
(\mu, \phi)\in D_{4-C_{0}A^{-1}}.
$
This completes the proof of Lemma 3.1.
\end{proof}
Repeating Lemma 3.1 n times, we have a new equation (still by $(\rho, \theta)$ denoting the variables, for brevity):
\begin{eqnarray}\label{3.12}
\dot{\rho}=f_{1}^{(n)}(\rho, \theta, t)+f_{2}^{(n)}(\rho, \theta, t),\!
\dot{\theta}=d A^{n}\rho^{2\beta-1}\!+g_{1}^{(n)}(\rho, \theta, t)\!+g_{2}^{(n)}(\rho, \theta, t),
(\rho, \theta, t)\in D_{4-CA^{-1}}\times \mathbb{T}^{1}\!,
\end{eqnarray}
where $f_{1}^{(n)},$ $f_{2}^{(n)}$ and $g_{1}^{(n)},$ $g_{2}^{(n)}$ satisfy \eqref{2.12} and \eqref{2.14}, respectively,
and
\begin{equation}\label{3.13}
f_{1}^{(n)}=f_{2}^{(n)}=O(A^{-1}),\;
g_{1}^{(n)}=O_{1}(A^{n-1}),\;\;g_{2}^{(n)}=O(A^{-1}).
\end{equation}
Let $F^{(n)}=f_{1}^{(n)}(\rho, \theta, t)+f_{2}^{(n)}(\rho, \theta, t).$ Then
$F(\rho, -\theta, -t)=-F(\rho, \theta, t),(\rho, \theta)\in D_{4-CA^{-1}},t\in \mathbb{T}^{1}, F=O(A^{-1}).
$
Rewrite \eqref{3.12} as follows
\begin{eqnarray}\label{3.20}
    \dot{\rho}=F(\rho, \theta, t),
    \dot{\theta}=d A^{n}\rho^{2\beta-1}+c_{0}h(\rho, t)+g_{1}(\rho, \theta, t)+g_{2}(\rho, \theta, t),
\end{eqnarray}
where $(\rho, \theta)\in D_{4-CA^{-1}},$ $t\in \mathbb{T}^{1},$ and
\begin{eqnarray}\label{3.21}
&&F(\rho, -\theta, -t)=-F(\rho, \theta, t),\;F=O(A^{-1}),\;
h(\rho, -t)=h(\rho, t),\;h=O_{1}(A^{n-1}),\\
\label{3.25}
&&g_{1}(\rho, -\theta, t)=g_{1}(\rho, \theta, t),\;g_{1}(\rho, \theta, -t)=g_{1}(\rho, \theta, t),\;
g_{2}(\rho, -\theta, -t)=g_{2}(\rho, \theta, t),\\
\label{3.26}
&&g_{1}=O_{1}(A^{n-1}),\;g_{2}=O(A^{-1}),\;
c_{0}=0.
\end{eqnarray}
\begin{lemma}
There exists a diffeomorphism
$\psi_{2}:\rho=\mu, \theta=\phi+U_{2}(\mu, \phi, t),(\mu, \phi, t)\in D_{4-C_{0}A^{-1}}\times \mathbb{T}^{1}$
such that $\psi_{2}(D_{4-C_{1}A^{-1}}\times \mathbb{T}^{1})\subset D_{4-C_{0}A^{-1}}\times \mathbb{T}^{1}, \;(C_{1}>C_{0}),$ and
\eqref{3.20} is transformed into
\begin{eqnarray}\label{3.28}
    \dot{\rho}=F^{(1)}(\rho, \theta, t),       \;
     \dot{\theta}=d A^{n}\rho^{2\beta-1}+h^{(1)}(\rho, t)+g_{1}^{(1)}(\rho, \theta, t)+g_{2}^{(1)}(\rho, \theta, t),
\end{eqnarray}
where $F^{(1)}$, $h^{(1)}$ satisfy \eqref{3.21}, $g^{(1)}_{1}, g^{(1)}_{2}$ satisfy \eqref{3.25}
and
$g_{1}^{(1)}=O_{1}(A^{n-2}),\;\;g_{2}^{(1)}=O(A^{-1}).$
\end{lemma}
\begin{proof}
Define a transformation
$\Phi_{2}:\mu=\rho,\phi=\theta+V_{2}(\rho, \theta, t),$
where\\
$V_{2}(\rho, \theta, t)=-\int_{0}^{\theta}\frac{g_{1}(\rho, s, t)-[g_{1}](\rho, t)}{d A^{n}\rho^{2\beta-1}+c\,h(\rho, t)}ds,
[g_{1}]=\int_{\mathbb{T}^{1}}g_{1}(\rho, s, t)ds.$
By \eqref{3.21}-\eqref{3.26}, we have
\begin{equation*}\label{3.30}
V_{2}(\rho, -\theta, t)=-V_{2}(\rho, \theta, t),\;\;V_{2}(\rho, \theta, -t)=V_{2}(\rho, \theta, t),\;\;V_{2}=O_{1}(A^{-1}).
\end{equation*}
Moreover, doing as in the proof of Lemma 1, we have that there exists $U_{2}=U_{2}(\mu, \phi, t)$ satisfying
\begin{equation}\label{3.30+}
U_{2}(\mu, -\phi, t)=-U_{2}(\mu, \phi, t),\;\;U_{2}(\mu, \phi, -t)=U_{2}(\mu, \phi, t),\;\;U_{2}=O_{1}(A^{-1}),
\end{equation}
and
$
\psi_{2}=\Phi_{2}^{-1}:\;\;\rho=\mu,\theta=\phi+U_{2}(\mu, \phi, t)
$
such that
$\psi_{2}(D_{4-C_{3}A^{-1}}\times \mathbb{T}^{1})\subset D_{4-C_{2}A^{-1}}\times \mathbb{T}^{1},\;\;C_{3}>C_{2}.$
Then \eqref{3.28} is changed into
$
\dot{ \mu}=F^{(1)}(\mu, \phi, t), \;
    \dot{\phi}=d A^{n}\mu^{2\beta-1}+h^{(1)}(\mu, t)+g_{1}^{(1)}(\mu, \phi, t)+g_{2}^{(1)}(\mu, \phi, t),
$
where
$F^{(1)}=F(\mu, \phi+U_{2}(\mu, \phi, t), t),$
$h^{(1)}(\mu, t)=c_{0}h(\mu, t)+[g_{1}](\mu, t),$
$g_{1}^{(1)}(\mu, \phi, t)=g_{1}(\mu, \phi+U_{2}(\mu, \phi, t), t)\frac{\partial V_{2}(\mu, \phi+U_{2}(\mu, \phi, t), t)}{\partial \theta},$
$g_{2}^{(1)}(\mu, \phi, t)=F^{(1)}(\mu, \phi+U_{2}(\mu, \phi, t), t)\frac{\partial V_{2}(\mu, \phi+U_{2}(\mu, \phi, t), t)}{\partial \rho}
+g_{2}(\mu, \phi+U_{2}(\mu, \phi, t), t)\frac{\partial V_{2}(\mu, \phi+U_{2}(\mu, \phi, t), t)}{\partial \theta}+\partial_{t}V_{2}(\mu, \phi+U_{2}(\mu, \phi, t), t).
$
By \eqref{3.21}-\eqref{3.26} and \eqref{3.30+}, we have that $F^{(1)},$ $h^{(1)},$ $g_{1}^{(1)},$ $g_{2}^{(1)}$ satisfy \eqref{3.21}, \eqref{3.25} and
$g_{1}^{(1)}=O_{1}(A^{n-2}),g_{2}^{(1)}=O(A^{-1}).$ This completes the proof of Lemma 2.
\end{proof}
Note that if $g(\rho, -\theta, t)=g(\rho, \theta, -t),$ we have $g(\rho, -\theta, -t)=g(\rho, \theta, t).$ Repeating Lemma 2 n times, we have that
\eqref{3.20} is changed into
\begin{equation}\label{3.50}
    \dot{\rho}=\mathcal{F}(\rho, \theta, t),    \;
    \dot{\theta}=d A^{n}\rho ^{2\beta-1}+H(\rho, t)+G(\rho, \phi, t),
\end{equation}
where $\mathcal{F}(\rho, -\theta, -t)=-\mathcal{F}(\rho, \theta, t),\mathcal{F}=O(A^{-1}),H(\rho, -t)=H(\rho, t),\;H=O_{1}(A^{n-1}),$
$G(\rho, -\theta, -t)=G(\rho, \theta, t), G=O(A^{-1}), (\rho, \theta, t)\in D_{4-C_{5}A^{-1}}\times \mathbb{T}^{1}, C_{5}>C_{4}.$
Let {$\lambda=\rho^{2\beta-1}+A^{-n}\int_{0}^{1}H(\rho, t)dt.$} By \eqref{3.50}, we get the time-1 map (refer to \cite{Liu1991,Liu1998})
\begin{equation}\label{p}
\mathcal{P}:\;
                    \lambda_{1}=\lambda_{0}+\xi (\lambda_{0}, \theta, A) ,\;
                   \theta_{1}=\theta_{0}+d A^{n}(\lambda+\eta (\lambda_{0}, \theta, A)),
(\lambda_{0}, \theta_{0})\in [2, 3]\times \mathbb{T}^{1},\;\;A\gg 1,
\end{equation}
where $\xi (\lambda_{0}, \theta_{0}, A),$ $\eta (\lambda_{0}, \theta_{0}, A)$ are analytic in $(\lambda_{0}, \theta_{0})\in [2, 3]\times \mathbb{T}^{1},$
$\xi (\lambda_{0}, -\theta_{0}, A)=-\xi (\lambda_{0}, \theta_{0}, A)$, $\eta (\lambda_{0}, -\theta_{0}, A)=\eta (\lambda_{0}, \theta_{0}, A)$,
$(\lambda_{0}, \theta_{0})\in [2, 3]\times \mathbb{T}^{1},$
$\sup_{(\lambda_{0}, \theta_{0})\in [2, 3]\times \mathbb{T}^{1}}|\xi (\lambda_{0}, \theta_{0}, A)|\leq C A^{-1},$ and
$\sup_{(\lambda_{0}, \theta_{0})\in [2, 3]\times \mathbb{T}^{1}}|\eta (\lambda_{0}, \theta_{0}, A)|\leq C A^{-1}.
$
\begin{lemma}\label{lem*}
Let $\Omega\subset \mathbb{{R}}^{m}$ be a closed ball of radius $1$ and $D_{*}\subset {C}^{m}$ be a complex neighbourhood of $\Omega.$ Let
$\tau_{0}, \widetilde{\tau}_{0}\in (0, 1].$ Let $B_{R}(b)$ also be a closed ball in $\mathbb{{R}}^{\kappa}$ with an arbitrary centre $b$ and radius
$R.$ Denote by $D$ the following domain in ${C}^{2m+\kappa}:$
$D=\{x\in {C}^{m}\mid |\mathrm{Im} x_{j}|< \tau_{0}\}\times \{y\in {C}^{m}\mid y\in D_{*}\}
\times \{\eta\in {C}^{\kappa}\mid |\eta_{t}-b_{t}|<R+\widetilde{\tau}_{0}\}.$
Suppose that $\gamma, c\in (0, 1]$ are fixed and on $D$ the following mappings are given:
$
A: (x, y, \eta)\mapsto (x+d \gamma A^{n}y+f^{1}(x, y, \eta),\; y+f^{2}(x, y, \eta),\; \eta+f^{3}(x, y, \eta)),$
$ G: (x, y, \eta)\mapsto (-x+\alpha^{1}(x, y, \eta),\;y+\alpha^{2}(x, y, \eta),\;\eta+\alpha^{3}(x, y, \eta)),
$
where $f^{\tau}$ and $\alpha^{\tau}$ $(\tau=1, 2, 3)$ are normal in $D$ functions.
Assume that $AGA=G$ throughout $D,$ where $AGA$ is defined, and if $\kappa>0$ then in addition $G^{2}=\mathrm{id}$ throughout
$D,$ where $G^{2}$ is defined. Let $K>0, \tilde{\varepsilon}>0.$ A number $\omega\in \mathbb{R}^{m}$ is called a number of type
 $\mathcal{M}_{m}(K, \tilde{\varepsilon})$
if for all $q\in \mathbb{Z}^{m}\setminus \{0\}$ and $\rho\in \mathbb{Z}$ such that $\mid \frac{(q, \omega)}{2\pi}-\rho\mid \geq \frac{K}{|q|^{m+\varepsilon}}.$
Introduce the notation $\Omega_{\gamma, C}=\{\omega\in \gamma d A^{n}\Omega\mid \omega $ {is of type} $\mathcal{M}_{m}(d A^{n}\gamma_{C}, 1)\}.$
Then for each $\varepsilon>0$ there exists $\delta>0$, depending only on $\varepsilon$, $D_{\tau}$ and $C$ but not on $\gamma,$ such
that if on $D$ $|f^{\tau}|<\gamma \delta$ and $|\alpha^{\tau}|<\gamma \delta$ then for each $\omega\in\Omega_{\gamma, C}$ the mappings $A$
and $G$ have a common invariant $(m+\kappa)$- dimensional manifold
\begin{equation}\label{*4}
x=\varphi+\Phi_{\omega}^{1}(\varphi, \chi),\;y=\gamma^{-1}\omega+\Phi_{\omega}^{2}(\varphi, \chi),\;\;\eta=\chi+\Phi_{\omega}^{3}(\varphi, \chi),
\end{equation}
where $\Phi_{\omega}^{\tau}$ are normal in
$\{\varphi\in {C}^{m} \big| |\mathrm{ Im \varphi_{j}}|<\frac{\tau_{0}}{2}\}\times \{\chi\in C^{\kappa}\big||\chi_{t}-b_{t}|<R+\frac{\widetilde{\tau}_{0}}{2}\}$
functions, such that diffeomorphisms of the manifold \eqref{*4} induced by the mappings $A$ and $G$ are $(\varphi, \chi)\mapsto (\varphi+\omega, \chi)$
and $(\varphi, \chi)\mapsto (-\varphi, \chi)$ respectively (so that \eqref{*4} is foliated into invariant under $A$ and $G$ spaces $\chi=const$) and the following inequality holds
$\Phi_{\omega}^{\tau}<\varepsilon.$
Moreover, for every two $\omega^{1}$ and $\omega^{2}$ in $\Omega_{\gamma, C}$ the following estimate holds
$|\Phi_{\omega^{1}}^{\tau}-\Phi_{\omega^{2}}^{\tau}|<\gamma^{-1}|\omega^{1}-\omega^{2}|\varepsilon.$
\end{lemma}
The present theorem is Theorem 1.1 of \cite{Sevryuk} when $d A^{n}=1.$ When $d A^{n}\neq 1,$ the proof is similar to that of Theorem 1.1 in \cite{Sevryuk} and so is omitted. See \cite{Sevryuk} for the details.

\ \

{\bf Proof of Theorem \ref{thm1}} Let $G: (\rho, \theta)\mapsto (\rho, -\theta)$ in Lemma \ref{lem*} and let $\eta$ vanish. By Lemma \ref{lem*} and \eqref{p}, $\mathcal{P}$ has an invariant curve $\mathcal{T}$ in the
annulus $[2, 3]\times \mathbb{T}^{1}.$ Since $A\gg 1$, it follows that the time-$1$ map of the original system has an invariant curve $\mathcal{T}_{A}$ in
$[2 A+C, 3A-C]\times \mathbb{T}^{1}$ with $C$ being a constant independent of $A.$ Choosing $A=A_{k}\rightarrow \infty$ as $k\rightarrow \infty,$
we have that there are countable many invariant curves $\mathcal{T}_{A_{k}},$ clustering at $\infty.$ Then any solution of the original system is bounded. Incidentally, we
can obtained that there are many infinite number of quasi-periodic solutions around the infinity in the $(x, \dot{x})$ plane.

\begin{thebibliography}{00}

\bibitem{Arnold1984}
V. I. Arnold, Reversible systems,
Nonlinear and turbulent processes in physics, Acad. Publ., New York, 1161-1174, 1984.

\bibitem{Arnold1986}
V. I. Arnold, M. B. Sevryuk, Oscillations and bifurcations in reversible systems,
 Nonlinear phenomena in plasma physics and hydrodynamics, Mir Publishers, Moscow, 31-64, 1986.

\bibitem{Dieckerhoff-Zehnder1987}
R. Dieckerhoff, E. Zehnder,
Boundedness of solutions via twist theorem,
{\it Ann. Scula. Norm. Sup. Pisa},
{\bf 14} (1), 79-95, 1987.

\bibitem{Graef}
 J. R. Graef, On the generalized Lienard equation with negative damping, {\it J. Diff. Eqs. } {\bf 12}, 34-62, 1972.

\bibitem{Laederich-Levi1991}
S. Laederich, M. Levi,
Invariant curves and time-dependent potential,
{\it Ergod. Th. and Dynam. Sys.},
{\bf 11} (2), 365-378, 1991.


\bibitem{Levinson1943}
N. Levinson, On the existence of periodic solutions for second order differential
equations with a forcing term, {\it J. Math. Phys.}, {\bf 32}, 41-48, 1943.



\bibitem{Liu1991}
B. Liu, An application of KAM theorem of reversible systems,
{\it Sci. Sinica Ser. A.}, {\bf 34} (9), 1068-1078, 1991.



\bibitem{Liu1998}
B. Liu, F. Zanolin, Boundedness of solutions of nonlinear differential equations,
{\it J. Diff. Eqs.} {\bf 144} (1), 66-98, 1998.






\bibitem{Moser}
J. Moser, Stable and random motion in dynamical systems: with special emphasis on celestial mechanics,
Princeton Uni. Press, 1973.

\bibitem{Piao2008}
D. Piao, W. Li, Boundedness of solutions for reversible system via Moser's twist theorem,
{\it J. Math. Anal. Appl.}, {\bf 341} (2), 1224-1235, 2008.

\bibitem{Reuter1951}
G. E. H. Reuter, A boundedness theorem for nonlinear differential equations of
the second order, {\it Proc. Cambridge Phil. Soc.}, {\bf 47}, 49-54, 1951.

\bibitem{Rong2001}
R. Yuan, X. Yuan, Boundedness of solutions for a class of nonlinear differential equations
 of second order via Moser's twist theorem, {\it Nonlinear
Anal.} {\bf 46} (8), 1073-1087, 2001.



\bibitem{Sevryuk}
M.B. Sevryuk, Reversible Systems, Lecture Notes in Math., vol. 1211, Springer, Berlin, 1986.



\bibitem{Yuan1995}
X. Yuan, Invariant tori of Duffing-type equations,
{\it Advances in Math. (China)},
{\bf 24} (4), 375-376, 1995.



\bibitem{Yuan1998}
X. Yuan,
Invariant eori of Duffing-type equations,
{\it J. Diff. Eqs.},
{\bf 142} (2), 231-262, 1998.

\bibitem{Yuan2000}
X. Yuan, Lagrange stability for Duffing-type equations, {\it J. Diff. Eqs.}, {\bf 160} (1), 94-117, 2000.

\bibitem{Yuan2017}
 X. Yuan, Boundedness of solutions for Duffing equation
with low regularity in time, {\it
Chin. Ann. Math. Ser. B.}
{\bf 38} (5), 1037-1046, 2017.

\end{thebibliography}
\end{document}